\documentclass[3p,review]{elsarticle}

\usepackage{amssymb}
\usepackage{amsthm}

\biboptions{square}

\journal{t.b.a.}

\usepackage{url}
\usepackage{dsfont}
\usepackage{amsmath}
\usepackage{verbatim}
\usepackage{color}

\newcommand{\vica}{\color{blue}}

\renewcommand{\leq}{\leqslant}
\renewcommand{\geq}{\geqslant}
\renewcommand{\Re}{\ensuremath{\operatorname{Re}}}
\renewcommand{\Im}{\ensuremath{\operatorname{Im}}}

\newcommand{\comp}{\mathds C}

\newcommand{\real}{\mathds{R}}

\newcommand{\Ee}{\mathds E}
\newcommand{\Pp}{\mathds P}
\newcommand{\I}{\mathds 1}

\newcommand{\psip}{p^{\mathsf U}}

\newcommand{\entier}[1]{\lfloor#1\rfloor}

\newtheorem{theorem}{Theorem}
\newtheorem{lemma}[theorem]{Lemma}
\newtheorem{proposition}[theorem]{Proposition}

\theoremstyle{definition}
\newtheorem{remark}[theorem]{Remark}
\newtheorem{definition}[theorem]{Definition}

\newtheorem{example}[theorem]{Example}

\begin{document}

\begin{frontmatter}



\title{On the small-time behaviour of L\'evy-type processes}

\author[kiev,dresden]{Victoria Knopova}
\author[dresden]{Ren\'e L.\ Schilling}

\address[kiev]{V.M.\ Glushkov Institute of Cybernetics, National Academy of Sciences of Ukraine, 40 Acad.\ Glushkov Ave., 03187 Kiev, Ukraine. \url{vicknopova@googlemail.com}}
\address[dresden]{Institut f\"ur Mathematische Stochastik, Technische Universit\"at Dresden, 01062 Dresden, Germany. \url{rene.schilling@tu-dresden.de}}

\begin{abstract}
    We show some Chung-type $\liminf$ law of the iterated logarithm results at zero for a class of (pure-jump) Feller or L\'evy-type processes. This class includes all L\'evy processes. The norming function is given in terms of the symbol of the infinitesimal generator of the process. In the L\'evy case, the symbol coincides with the characteristic exponent.
\end{abstract}

\begin{keyword}
    Law of the Iterated Logarithm\sep
    small-time asymptotic\sep
    L\'evy process\sep
    Feller process\sep
    L\'evy-type process\sep
    symbol\sep
    pseudo differential operator\sep
    stochastic differential equation.

    \MSC[2010] 60F15 \sep 60G51 \sep 60J35 \sep 60J75
\end{keyword}
\end{frontmatter}



\section{Introduction}\label{intro}
We study the short-time behaviour of a class of one-dimensional Feller processes $(X_t)_{t\geq 0}$. To do so we identify suitable norming functions $u,v,w$ such that the following Chung-type LIL (law of the iterated logarithm) assertions hold $\Pp^x$-almost surely:
\begin{gather}
\label{intro-e01}
    \varliminf_{t\to 0} \frac{\sup_{0\leq s\leq t}|X_s-x|}{u^{-1}(x,{t}/{\log|\log t|})}=C(x),\\
\label{intro-e02}
    \varlimsup_{t\to 0} \frac{\sup_{s\leq t}|X_s-x|}{v(t,x)}=0 \quad\text{or\ } =+\infty,\\
\label{intro-e03}
    \varliminf_{t\to 0} \frac{|X_t-x|}{w(t,x)} = \gamma(x)>0 \quad\text{or}\quad =+\infty.
\end{gather}
Assertions of this kind are classical for Brownian motion, the corresponding results for L\'evy processes are due to Dupuis \cite{Du74} and Aurzada, D\"{o}ring and Savov \cite{ADS12}. The class of Feller processes considered in this paper includes L\'evy processes and extends the results of these authors. We will characterize the norming functions with the help of the symbol of the infinitesimal generator of the Feller process. In the case of a L\'evy process this becomes a rather simple criterion in terms of the characteristic exponent of the process.

\paragraph{L\'evy processes}
A (real-valued) L\'evy process $(X_t)_{t\geq 0}$ is a stochastic process with stationary and independent increments and c\`adl\`ag (right continuous with finite left limits) sample paths. The transition function is uniquely determined through the characteristic function which is of the following form:
$$
   \lambda_t(x,\xi):= \Ee^x e^{i\xi (X_t - x)}
    = \Ee^0 e^{i\xi X_t}
    = e^{-t\psi(\xi)},
    \quad t\geq 0,\quad  \xi\in\real.
$$
The \emph{characteristic exponent} $\psi:\real\to\comp$ is given by the \emph{L\'evy-Khintchine formula}
\begin{equation}\label{intro-e05}
    \psi(\xi)
    = i l\xi + \tfrac 12\,\sigma^2\xi^2 + \int_{\real\setminus\{0\}} \big(1-e^{iy\xi}+iy\xi\I_{(0,1]}(|y|)\big)\,\nu(dy)
\end{equation}
and the \emph{L\'evy triplet} $(l,\sigma^2,\nu)$ where $\nu$ is a measure on $\real\setminus\{0\}$ such that $\int_{y\neq 0} (1\wedge y^2)\,\nu(dy) < \infty$, and $l\in\real, \sigma\geq 0$. The characteristic exponent is also the symbol of the infinitesimal generator $A$ of the L\'evy process:
\begin{equation*}
    Au(x)
    = -\psi(D)u(x)
    := -\int_\real e^{ix\xi} \hat u(\xi)\,\psi(\xi)\,d\xi,\quad u\in C_c^\infty(\real),
\end{equation*}
where $\hat u(\xi) = (2\pi)^{-1} \int_\real u(x) e^{-ix\xi}\,dx$ denotes the Fourier transform of $u$.

\paragraph{Feller processes}
The generator  of a L\'evy process has \emph{constant coefficients}: it does not depend on the state space variable $x$. This is due to the fact that a L\'evy process is spatially homogeneous which means that the transition semigroup $P_tu(x) = \Ee^x u(X_t) = \Ee u(X_t+x)$ is given by convolution operators. We are naturally led to \emph{Feller processes} if we give up spatial homogeneity.
\begin{definition}\label{intro-10}
    A (one-dimensional) \emph{Feller process} is a real-valued Markov process $(X_t)_{t\geq 0}$ whose transition semigroup
    $P_t u(x) := \Ee^x u(X_t)$, $u\in B_b(\real)$, is a \emph{Feller semigroup}, i.e.
    \begin{enumerate}[a)]
    \item $P_t$ is Markovian: if $u\in B_b(\real), u\geq 0$ then $P_t u\geq 0$ and $P_t 1 = 1$;
    \item $P_t$ maps $C_\infty(\real) := \big\{u\in C(\real)\,:\, \lim_{|x|\to\infty} u(x) = 0\big\}$ into itself;
    \item $P_t$ is a strongly continuous contraction semigroup in $(C_\infty(\real),\|\cdot\|_\infty)$.
    \end{enumerate}
\end{definition}
Every L\'evy process is a Feller process.

Write $(A,D(A))$ for the generator of the Feller semigroup. If $C_c^\infty(\real)\subset D(A)$, then
\begin{equation*}
    Au(x)
    = -p(x,D)u(x)
    := -\int_\real e^{ix\xi} \hat u(\xi)\,p(x,\xi)\,d\xi,\quad u\in C_c^\infty(\real),
\end{equation*}
see e.g.\ \cite[Vol.~1, Theorem~4.5.21, p.~360]{Jac}; this means that $A$ is a pseudo differential operator whose symbol $p:\real\times\real\to\comp$ is such that for every fixed $x$ the function $\xi\mapsto p(x,\xi)$ is the characteristic exponent of a L\'evy process
\begin{equation}\label{intro-e25}
    p(x,\xi)
    = i l(x)\xi + \tfrac 12\, \sigma^2 (x)\xi^2 + \int_{\real\setminus \{0\}} \big( 1- e^{i \xi y}+ i\xi y\I_{(0,1]}(|y|) \big)\, \nu(x,dy).
\end{equation}
The L\'evy triplet $(l(x),\sigma^2(x),\nu(x,dy))$ now depends on the state space, i.e.\ the generator is an operator with variable `coefficients'. Typical examples are elliptic diffusions where the symbol (in one dimension) is of the form $p(x,\xi) = \frac 12 \sigma^2(x)\xi^2$ and stable-like processes where $p(x,\xi) = |\xi|^{\alpha(x)}$ with $0<\alpha_0\leq \alpha(x)\leq \alpha_1 < 2$ is Lipschitz continuous, cf.\ \cite{bass}. For further details we refer to \cite{Jac} or \cite{jac-sch-survey}.

The symbol $p(x,\xi)$ plays very much the same role as the characteristic exponent of a L\'evy process and it is possible to use $p(x,\xi)$ to describe the path behaviour of a Feller process, for example \cite{Sch98}, \cite{jac-sch-survey} or \cite{SW12}. Note however that, due to the lack of spatial homogeneity, $p(x,\xi)$ is not the exponent of the characteristic function, i.e.
\begin{equation*}
    \lambda_t(x,\xi)
    = \Ee^x e^{i(X_t-x)\xi}  \neq e^{-tp(x,\xi)}.
\end{equation*}

\paragraph{A brief overview of LIL-type results}
 For a general L\'evy process the first result is due to Khintchine \cite{Kh39}, cf.\ \cite{Kh36} for the Brownian LIL. Khintchine provides a necessary and sufficient criterion for a positive increasing function $u: (0,\epsilon)\to (0,\infty)$ to be the upper function for a one-dimensional L\'evy process $(X_t)_{t\geq 0}$ without Gaussian component:
\begin{equation}\label{intro-e50}
    \varlimsup_{t\to 0} \frac{|X_t|}{u(t)} \leq c
    \quad\Pp^0\text{-a.s.\ if, and only if,}\quad
    \int_{0+} \frac{\Pp^0 \{ |X_t|> c u(t)\}}{t}\,dt<\infty.
\end{equation}
As usual, we indicate by $\int_{0+}\ldots$ that the integral converges at the origin. For a Brownian motion this result is sharp with $u(t)=\sqrt{t \log|\log t|}$.

Khintchine's result is generalized by the following integral test due to Savov \cite{Sa09}. Let $N(t):= \int_{|x|>t} \nu(dx)$ and $b(t)$ be a function which satisfies some mild growth assumptions. Then
\begin{align*}
    \int_0^1 N(b(t))dt<\infty\quad\text{or\ } =+\infty
    \implies
    \varlimsup_{t\to 0} \frac{|X_t|}{b(t)}=\lambda(b) \quad\text{or\ } =+\infty.
\end{align*}

The first Chung-type LIL for ($n$-dimensional $\alpha$-stable) L\'evy processes is due to Taylor \cite{Ta67}. If $0<\alpha<n$ and if the transition density satisfies $p_t(0)>0$, then \eqref{intro-e01} holds with $u^{-1}(x,t) = t^{1/\alpha}$ and $C(x)=C$.
Pruitt and Taylor \cite{PT69} extended this result for L\'evy processes with independent stable components.
Based on \cite{Ha71}, Fristedt and Pruitt \cite{FP71,FP72}
prove a LIL for subordinators (one-sided increasing L\'evy processes), where the upper function is determined by the Laplace exponent of the process.
Dupuis \cite{Du74} extends these results for \emph{symmetric} L\'evy processes, with $u^{-1}(t):= 1/\psi^U (1/t)$, where $\psi^U(\xi)= \int_{y\neq 0} \min\{ 1, |\xi u|^2\}\nu(du)$.
 Using a different approach, this result  was independently rediscovered by Aurzada--D\"oring--Savov \cite{ADS12}.

\section{A Chung-type $\liminf$ LIL for Feller processes}\label{sec-chung}

Consider a one-dimensional Feller process $(X_t)_{t\geq 0}$ with symbol $p(x,\xi)$ of the form \eqref{intro-e25}. Throughout we assume:
\begin{gather}\tag{\bfseries A1}\label{A1}\begin{gathered}
    C_c^\infty(\real)\quad\text{is in the domain of the infinitesimal generator};\\
    x\mapsto p(x,\xi)\quad\text{is continuous and has no diffusion part:\ \ } \sigma^2\equiv 0; \\
    \textit{sector condition:\ \ }
    \exists c_0\in (0,\infty)\quad
    \forall x,\xi\in\real
    \::\: |\Im  p(x,\xi)|\leq c_0 \Re  p(x,\xi).
\end{gathered}\end{gather}
Define the function
\begin{equation}\label{ps}
    \psip(x,\xi):= \int_{y\neq 0} \min\{|\xi y|^2, \, 1\}\,\nu(x,dy).
\end{equation}
It is not hard to see that $|p(x,\xi)| \leq 2\psip(x,\xi)$ and
$\psip(x,2\xi)\leq 4\psip(x,\xi)$ for all $x, \xi \in \real$.
We will also need the following regularity assumptions:
\begin{gather}
\tag{\bfseries A2}\label{A5}
    \exists \kappa(x)>1\quad\forall R\leq 1 \::\:
    \sup_{|x-y|\leq     2R   } \psip\big(y,\tfrac{1}{R}\big)
    \leq \kappa(x)  \inf_{|x-y|\leq    3 R   } \psip\big(y,\tfrac{1}{R}\big);\\
\tag{\bfseries A3}\label{A6}
    \exists t_0\in (0,1)\;\quad \exists q=q(x)\in (0,1)\;\quad \forall R\in (0,1],\; y\in B(x,R),\; t\in [0,t_0]\::\: \Pp^y(X_t < y) \leq  q.
\end{gather}
For example \eqref{A6} holds (even with equality) with $q=1/2$ if $\lambda_t(x,\xi) = \Ee^x e^{i\xi(X_t-x)}$ is real-valued. For L\'evy processes which are not compound Poisson processes  \eqref{A6}  follows from $\lim_{t\to 0}\Pp(X_t>0) =\rho\in (0,1)$ (Spitzer's condition); see \cite[Chaper 7]{Doney} for the necessary and sufficient conditions in the L\'evy case. Set
\begin{equation}\label{chung-e05}
    u\equiv u(x,R):= \frac{1}{\inf_{|x-y|\leq 3 R} \psip\big(y,\tfrac{1}{R}\big)}, \quad R\in (0,1],
\end{equation}
and denote by $u^{-1}(x,\rho) := \inf\{ r:\,\, u(x,r)\geq \rho \}$ the generalized inverse of $R\mapsto u(x,R)$.

We can now state the main result of this section.
\begin{theorem}\label{chung-10}
    Let $(X_t)_{t\geq 0}$ be a one-dimensional Feller process with symbol $p(x,\xi)$ satisfying \eqref{A1}--\eqref{A6}. Then there exists a constant $C(x)>0$ such that
    \begin{equation}\label{chung-e10}
        \varliminf_{t\to 0} \frac{\sup_{0\leq s\leq t}|X_s-x|}{u^{-1}(x,{t}/{\log|\log t|})}=C(x)
        \qquad(\Pp^x\text{-a.s.})
    \end{equation}
    where $u^{-1}$ is the generalized inverse of the function $R\mapsto u(x,R)$ defined in \eqref{chung-e05}.
\end{theorem}

Before we prove Theorem~\ref{chung-10} let us consider an example.
\begin{example}\label{chung-20}
    Take $\nu(x,dy)= \frac 14\,{\alpha(x)(2-\alpha(x))}\,{|y|^{-1-\alpha(x)} }\, dy$, where  $\alpha: \real\to [\alpha_0,\alpha_1]\subset(0,2)$ is continuously differentiable, with uniformly bounded derivative.  Clearly, \eqref{A1} holds. A direct calculation shows that $\psip(x,\xi)= |\xi|^{\alpha(x)}$.

     We will now check \eqref{A5}. Pick $R\in (0,1]$. Since $\alpha$ is continuously differentiable, we have
    $$
    \sup_{|x-y|\leq 3R} \alpha(y)- \inf_{|x-y|\leq 3R} \alpha(y)
    = \max_{z,y\in B(x,3R)} |\alpha(y)-\alpha(z)| \leq 6R\max_{y\in B(x,3R)}|\alpha'(y)|,
    $$
    implying
    \begin{align*}
        \frac{\sup_{|x-y|\leq 2R} \psip(y,1/R)}{\inf_{|x-y|\leq 3R} \psip(y,1/R)}
        \leq \frac{\sup_{|x-y|\leq 3R} \psip(y,1/R)}{\inf_{|x-y|\leq 3R} \psip(y,1/R)}
        \leq \Big(\frac1{R}\Big)^{6R\max\limits_{y\in B(x,3R)}|\alpha'(y)|}
        \leq \Big(\frac1{R^R}\Big)^{6\max\limits_{y\in B(x,3R)}|\alpha'(y)|}
    \end{align*}
    for all $R\in (0,1]$. Thus, \eqref{A5} holds with any function $\kappa(x)\geq 64^{\max_{y\in B(x,1)}|\alpha'(y)|}$.

    Let us check \eqref{A6}. In \cite[Theorem~5.1]{Ko00} it is proved that for $\alpha\in C^1_b(\real)$, there exists a Feller process $(X_t)_{t\geq 0}$ corresponding to the characteristic triplet $\big(0,0,\frac 14\,{\alpha(x)(2-\alpha(x))}\,{|y|^{-1-\alpha(x)} }\, dy\big)$. Moreover, this process has a transition density $p(t,y,z)$, and
    $$
        p(t,y,z)
        = p_{\alpha(y)}(t,y-z)[1+O(1)\min\{1, (1+|\log t|)|y-z|\} + O(t^\delta)] + \frac{O(t)}{1+|y-z|^{\alpha_0+1}}
    $$
    for the symmetric $\alpha(y)$-stable transition density $p_{\alpha(y)}(t,y-z)$ and some $\delta\in (0,1)$; the big-$O$ terms refer to $t\to 0$ and do not depend on $y,z$. Using the scaling property $p_{\alpha(y)}(t,y-z) = t^{-\alpha(y)} p_1(1,t^{-1/\alpha(y)}(y-z))$ and the unimodality of the stable law we get
    $$
        \Pp^y(X_t>y)
        \geq \int_{y\leq z\leq y+ \epsilon t^{1/\alpha(y)}}  p(t,y,z)\,dz
        \geq \epsilon p_{\alpha(y)}(1,\epsilon)(1+O(1)) + O(t^{1/\alpha(y)})
    $$
    which proves \eqref{A6}.

    Next we calculate the rate of convergence. Since $\alpha(y)$ is  continuously differentiable, we may assume that on $(x-3R,x+3R)$ there is a local minimum at $x$, say. Then for $R$ small enough $\min_{|x-y|\leq 3R} \alpha(y)=\alpha(x)$, and so
    $$
        u(x,R)
        =\left(\inf_{|x-y|\leq 3R} \psip(y,R^{-1})\right)^{-1}
        =\sup_{|x-y|\leq 3R} R^{\alpha(y)}= R^{\min_{|x-y|\leq 3R} \alpha(y)}
        = R^{\alpha(x)}.
    $$
    Consequently, $u^{-1}(x,\rho)= \rho^{1/\alpha(x)}$.

    Assume now that $x$ is not a local minimum of $\alpha$. Then $x$ is either a local maximum, or $\alpha$  is decreasing (respectively,  increasing) on $[x-3R,x+3R]$. In both cases the minimum is attained at one of the endpoints.
    Without loss of generality we assume that the minimum is attained at the point $x-3R$.  Thus,
    $$
        u(x,R)=    R^{\alpha(x- 3R)}
    $$
    and as $\alpha_0\leq \alpha(x)\leq \alpha_1$, we have
    \begin{equation}\label{rho01}
        c_0 \rho^{1/\alpha_0}\leq   u^{-1}(x,\rho)\leq c_1 \rho^{1/\alpha_1}.
    \end{equation}
    Since $\alpha$ is continuously differentiable we get, using a Taylor expansion,
    \begin{equation}\label{al}
        \alpha(x- R)= \alpha(x) - R\alpha'(x -  \theta R),
    \end{equation}
    where $\theta=\theta(x,R)\in (0,1)$. Note that the function $ R\mapsto R^{\alpha(x-3R)}$ is continuous and tends to $0$ as $R\to 0$, implying that for sufficiently small $\rho$ the equation $R^{\alpha(x-3R)}= \rho$ admits a solution; thus,
    $$
        u^{-1}(x,\rho) = \min\{ R \,:\, R^{\alpha(x-3R)}= \rho\}.
    $$
    By \eqref{al} the function $u^{-1}(x,\rho)$ satisfies the equation
    \begin{equation}\label{u-eq}
        u^{-1}(x,\rho)= \rho^{1/(\alpha(x)- 3u^{-1}(x,\rho) \alpha'(x-3\theta u^{-1}(x,\rho)))}.
    \end{equation}
    Therefore, by \eqref{u-eq}  we have
    \begin{equation*}
    \begin{split}
        \lim_{\rho\to 0} \frac{\rho^{1/\alpha(x)}}{u^{-1}(x,\rho)}
        &=\lim_{\rho\to 0} \exp \Big\{\Big(\frac{1}{\alpha(x)}- \frac{1}{\alpha(x)- 3u^{-1}(x,\rho) \alpha'(x-3\theta u^{-1}(x,\rho))}\Big)\ln \rho\Big\}
        =1,
    \end{split}
    \end{equation*}
    where we used that $\alpha(x) \in [\alpha_0,\alpha_1]$ and, because of \eqref{rho01}, $u^{-1}(x,\rho) \ln \rho\to 0$ as $\rho\to 0$.

  This gives \eqref{chung-e10} with $u^{-1}(x,\rho)= \rho^{1/\alpha(x)}$.
\end{example}

For the proof of Theorem~\ref{chung-10} we need several auxiliary results in order to estimate the probability that $(X_t)_{t\geq 0}$ exits a ball of radius $r>0$ within time $t>0$.

Lemma~\ref{chung-30} is the key to derive the LIL. We record it in a form which is convenient for our purposes, and refer to \cite{Sch98} for the original version, as well as to its improvement (with a simplified proof) \cite[Proposition~4.3]{SW12}. 
A close inspection of the arguments in \cite{SW12} reveals that one does not need the `bounded coefficients' assumption $\sup_{x\in\real, |\xi|\leq 1}|p(x,\xi)| < \infty$.
\begin{lemma}\label{chung-30}
    Let $(X_t)_{t\geq 0}$ be a one-dimensonal Feller process with symbol $p(x,\xi)$ satisfying \eqref{A1}. Then for all $t>0$ and $R>0$  we have
    \begin{gather}\label{chung-e31}
        \Pp^x \Big(\sup_{s\leq t}|X_s-x| \geq R\Big)
        \leq c\, t   \sup_{|x-y|\leq R}  \psip\big(y,\tfrac{1}{R}\big),\\
    \label{chung-e33}
        \Pp^x \Big(\sup_{s\leq t}|X_s-x| < R\Big)
        \leq c\left( t  \inf_{|x-y|\leq R} \psip\big(y,\tfrac{1}{ R}\big)\right)^{-1}.
    \end{gather}
    The constant $c\geq 1$ depends only on the sector constant $c_0$, but not on $x$.
\end{lemma}
First we extend \eqref{chung-e33}.

\begin{lemma}\label{chung-40}
    Under the assumptions of Lemma~\ref{chung-30}, we find for $n\geq 2$ and all $t,R>0$
    \begin{equation}\label{chung-e35}
        \Pp^x \Big( \sup_{s\leq nt}|X_s - x| < R\Big)
        \leq (4c)^n\left( t \inf_{|x-y|\leq   3 R }  \psip\big(y,\tfrac{1}{ R}\big)\right)^{-n}.
    \end{equation}
\end{lemma}
\begin{proof}
Set, for simplicity, $X_t^* := \sup_{s\leq t}|X_s-x|$.
We use induction and the Markov property.
\begin{align*}
    \Pp^x (X_{nt}^*<R)
    &\leq \Ee^x \left( \I_{\{X_{(n-1)t}^*<R\}} \I_{\{\sup_{0\leq s\leq t} |X_{(n-1)t+s}-X_{(n-1)t}|<2R\}}\right)\\
    &= \Ee^x \left( \I_{\{X_{(n-1)t}^*<R\}} \Ee^{X_{(n-1)t}} \big[\I_{\{ X_t^*<2R\}}\big]\right)\\
    &\leq \sup_{|x-y|\leq R} \Ee^y\big[ \I_{\{X_t^*<2R\}} \big] \Ee^x \big[\I_{\{ X_{(n-1)t}^* < R\}}\big]\\
    &\leq  c\, \sup_{|x-y|\leq R} \left( t \inf_{|z-y|\leq  2R } \psip\big(z,\tfrac{1}{2 R}\big) \right)^{-1}
                         (4c)^{n-1} \left( t \inf_{|x-z|\leq  3R } \psip\big(z,\tfrac{1}{ R}\big)\right)^{-n+1}\\
    &\leq (4c)^n\left( t  \inf_{|x-y|\leq  3 R  } \psip\big(y,\tfrac{1}{ R}\big) \right)^{-n},
\end{align*}
where we used, in the second line from below, the fact that $\psip(x,2\xi)\leq 4\psip(x,\xi)$.
\end{proof}

\begin{remark}\label{chung-50}
    Let $u(x,R)$ be as in \eqref{chung-e05}. Then \eqref{chung-e35} becomes for any $\gamma>1$
    \begin{equation*}
        \Pp^x \Big(\sup_{s\leq n\cdot (4\gamma c) u(x,R)}|X_s-x| <R\Big) \leq \gamma^{-n},
        \qquad R>0.
    \end{equation*}
\end{remark}
\begin{lemma}\label{chung-60}
    Suppose that the assumptions of Theorem~\ref{chung-10} are satisfied. Denote by $c$ the constant appearing in Lemma~\ref{chung-30}, by $\kappa(x)$ the constant from \eqref{A5}, and by $\gamma=\gamma(x)> \max\big\{ 1, \frac{\kappa(x)}{4(1-q)}\big\}$, where $q=q(x)$ is the constant from \eqref{A6}.  Then there exist constants $p_{1,\gamma}(x)$, $p_{2,\gamma}(x)\in (0,1)$  such that  for all $m\geq 1$
    \begin{equation}\label{chung-e45}
        p_{2,\gamma}(x)^{m+1} \leq \Pp^x \Big(\sup_{s\leq mu(x,R)}|X_s-x|\leq R\Big)  \leq p_{1,\gamma}(x)^m.
    \end{equation}
\end{lemma}
\begin{proof} Set $X_t^* := \sup_{r\leq t}|X_r-x|$ and $X_{s,t}^* := \sup_{s\leq r\leq t}|X_r-X_s|$.  First we prove
    \begin{gather}\label{chung-e47}
        C_2^n \leq \Pp^x (X_{n u(x,R)/(16\kappa(x)\gamma c)}^* \leq R)
        \qquad\text{and}\qquad
        \Pp^x (X_{n\cdot (4\gamma c) u(x,R)}^* \leq R) \leq C_1^n,
    \end{gather}
    where $n\geq 1$ and $C_1, C_2>0$ are some constants. The upper bound follows from Lemma~\ref{chung-40} and Remark~\ref{chung-50} with $C_1 = 1/\gamma$, where $\gamma>1$ is arbitrary, independent of $x$.

    Let us establish the lower bound. The crux of the matter is now the behaviour of $\psip(x,1/R)$ with respect to the variable $x$. Recall that $\psip(x,\xi)$ satisfies \eqref{A5} with some constant $\kappa(x)$.  Then we get from \eqref{chung-e31} and \eqref{A5} for any $z$ such that $|z-x|<R$
    \begin{align*}
        \Pp^z (X_{u(x,R)/(4c\gamma)}^*\geq R)
        \leq  \frac{1}{4\gamma}\, u(x,R) \sup_{|z-y|\leq R} \psip\big(y,\tfrac{1}{ R}\big)
        \leq  \frac{1}{4\gamma}\, u(x,R) \sup_{|x-y|\leq 2R} \psip\big(y,\tfrac{1}{ R}\big)
        \leq  \frac{\kappa(x)}{4\gamma}.
    \end{align*}

    Taking  $\gamma\equiv \gamma(x)>\max\big\{1,\,\frac{\kappa(x)}{4(1- q)}\big\}$,  we find for all $z$ with $|z-x|<R$
    \begin{equation}\label{chung-e50}
        \Pp^z (X^*_{u(x,R)/(4 c \gamma)} <  R) \geq  1-\frac{\kappa(x)}{4 \gamma}>q,
        \qquad R>0.
    \end{equation}
    Let $T:= u(x,R)/(4\gamma c)$ be fixed. Observe that
    $
        \{ X_{nT}^* \leq 2R\} \supset \bigcap_{k=0}^{n-1}A_k,
    $
    where
    \begin{equation*}
    A_k:=\left\{ X_{kT,(k+1)T}^*\leq R, \quad
                X_{(k+1)T}-X_{kT}\in
                \left\langle\begin{aligned}
                &[0,R], \text{\ \ if\ \ } X_{kT}<x,\\[-\medskipamount]
                &[-R,0], \text{\ \ if\ \ } X_{kT}\geq x.
                \end{aligned}\right.
        \right\}
\end{equation*}
In other words, if $X_{kT}<x$, then at the next end-point $(k+1)T$ the process is above $X_{kT}$, but within the ball $B(X_{kT}, R)$, and if $X_{kT} \geq x$, then at time  $(k+1)T$ the process is below $X_{kT}$ but still in the ball $B(X_{kT}, R)$.
Denote $\mathcal{F}_k:= \sigma\{ X_s$, $ s\leq Tk\}$. Then
\begin{equation*}
\begin{split}
    \Ee^x \left[ \I_{A_{n-1}} | \mathcal{F}_{n-1}\right]
    &= \Pp^{X_{(n-1)T}} (A_{n-1})\\
    &= \Pp^{X_{(n-1)T}} \left(X_T^*\leq R,\: X_{T}-X_0\in [0,R]\right) \I_{\{X_{(n-1)T} < x\}}\\
    &\qquad\mbox{}+ \Pp^{X_{(n-1)T}} \left(X_T^*\leq R,\: X_{T}-X_0\in [-R,0]\right)\I_{\{X_{(n-1)T} \geq x\}}\\
    &\geq \inf_{x-R\leq z<x} \Pp^z \left(X_T^*\leq R,\: X_{T}-X_0\in [0,R]\right).
 \end{split}
 \end{equation*}
 Without loss of generality we may assume that \eqref{A6} holds with $t_0=1$; otherwise we would just get a further multiplicative factor. 
By \eqref{A6}  we have
\begin{align*}
    \Pp^z \left(X_T^*\leq R,\: X_{T}-X_0\in [-R,0]\right)
    \leq  \Pp^z \left( X_{T}-X_0<0\right)
    \leq q,
\end{align*}
uniformly in $z\in B(x,R)$ and $T\in [0,1]$. Using \eqref{chung-e50} with $z\in B(x, R)$ we get
\begin{align*}
    \Pp^z \left(X_T^*\leq R,\: X_{T}-X_0\in [0,R]\right)
    &= \Pp^z \left(X_T^*\leq R\right) - \Pp^z \left(X_T^*\leq R,\: X_{T}-X_0\in [-R,0]\right)\\
    & \geq \Pp^z \left(X_T^*\leq R\right) -q
    \geq C_2,
\end{align*}
where $C_2:= 1 - \frac{\kappa(x)}{4\gamma}-q>0$ by our choice of $\gamma$. Thus,
\begin{equation*}
    \Ee^x \left[ \I_{A_{n-1}} | \mathcal{F}_{n-1}\right] \geq C_2.
\end{equation*}
Note that $\prod_{k=0}^{n-2} \I_{A_k} $ is $\mathcal{F}_{n-1}$-measurable, and by the Markov property,
\begin{align*}
    \Ee^x\bigg( \prod_{k=0}^{n-1} \I_{A_k}\bigg)
    &= \Ee^x\bigg( \Ee^x \bigg[ \prod_{k=0}^{n-1} \I_{A_k}\:\bigg|\: \mathcal{F}_{n-1}\bigg]\bigg)
    = \Ee^x\bigg(\prod_{k=0}^{n-2} \I_{A_k} \Ee^x \bigg[ \I_{A_{n-1}}\:\bigg|\: \mathcal{F}_{n-1}\bigg]\bigg)
    \\
    &= \Ee^x\bigg( \prod_{k=0}^{n-2} \I_{A_k} \Pp^{X_{(n-1)T}} (A_{n-1})\bigg)
    \geq C_2\Ee^x\bigg( \prod_{k=0}^{n-2} \I_{A_k} \bigg).
\end{align*}
With \eqref{A5} and the fact that $\psip(x,2\xi)\leq 4\psip(x,\xi)$ we see
\begin{align*}
    \inf_{|x-y|\leq 3R} \psip(y,1/R)
    &\leq \inf_{|x-y|\leq 2R} \psip(y,1/R)
    \leq \sup_{|x-y|\leq 2R} \psip(y,1/R)\\
    &\leq \kappa (x)\inf_{|x-y|\leq 6R} \psip(y,1/R)
    \leq 4\kappa (x)\inf_{|x-y|\leq 6R}  \psip(y,1/(2R)),
\end{align*}
which implies that $u(x,2R)\leq  4 \kappa(x) u(x,R)$. Thus, by induction (recall that $T=u(x,R)/(4\gamma c)$)
\begin{equation*}
    \Pp^x (X_{ n u(x,2R)/(16\kappa(x)\gamma c)}^*\leq 2R)
    \geq  \Pp^x (X_{nu(x,R)/(4\gamma c)}^*\leq 2R)
    = \Pp^x (X_{nT}^*\leq 2R)
    \geq \Ee^x\left[ \prod_{k=0}^{n-1} \I_{A_k}\right]\geq C_2^n.
\end{equation*}

Finally, we show how \eqref{chung-e45} follows from \eqref{chung-e47}.
Put $m:= \entier{n(4\gamma c)}+1$ ($\entier{x}$ denotes the largest integer smaller or equal to $x\in\real$); then $n\cdot (4\gamma c)\leq m \leq n\cdot (4\gamma c)+1$, implying
\begin{align*}
        \Pp^x (X_{mu(x,R)}^* \leq R )
        \leq \Pp^x ( X_{n(4\gamma c)u(x,R)}^* \leq R)
        \leq C_1^n= (C_1^{\frac{n}{m}})^m
        \leq C_1^{\frac{m}{4\gamma c+1}}=:p_{1,\gamma}^m(x).
\end{align*}
For the lower bound we set $m:=\entier{n/(16\gamma \kappa(x)c)}$. Then $\frac{n}{16\gamma\kappa(x) c}-1\leq m\leq \frac{n}{16\gamma \kappa(x) c}$, and
\begin{gather*}
        \Pp^x (X_{mu(x,R)}^*\leq R) \geq \Pp^x (X_{n u(x,R)/(16\gamma\kappa(x) c)}^* \leq R)
        \geq C_2^n
        = (C_2^{\frac{n}{m+1}})^{m+1}
        \geq C_2^{(m+1)16\gamma\kappa(x) c}=: p_{2,\gamma}^{m+1}(x).
\tag*{\qedhere}
\end{gather*}
\end{proof}
\begin{remark}\label{chung-70}
    Note that $p_{1,\gamma}$, $p_{2,\gamma}$ in Lemma~\ref{chung-60} depend on $x$  through $\kappa(x)$ and $\gamma\geq \max\big\{1, \,\frac{\kappa(x)}{4(1-q(x))}\big\}$. Without loss of generality we may choose them to be continuous in $x$.
\end{remark}
\begin{proof}[Proof of Theorem~\ref{chung-10}]  Fix $x\in \real$ and write $\tau^x(a) := \inf\{s\geq 0\,:\, X_s-x\notin [-a,a]\}$ for the first exit time of the process $(X_t)_{t\geq 0}$ with $X_0=x$. Then we can follow the arguments from \cite{Du74}. Note that we can replace the stationary and independent increments assumption in \cite{Du74} by the strong Markov property and the fact that the constants $p_{1,\gamma}, p_{2,\gamma}$ in \eqref{chung-e45} depend continuously on $x$. Using \eqref{chung-e45} we can prove, as in \cite{Du74}, that there exists a constant $\xi\in (0,\infty)$ such that
    \begin{equation}\label{chung-e60}
        \Pp^x\left(\sup_{2a_{2m} \leq a\leq 2a_m}  \frac{\tau^x(a)}{u(x,a) \log|\log u(x,a)|}<\xi\right) \leq e^{-m^{1/4}},\quad m\geq 1,
    \end{equation}
    where $a_m= a_m(x) $ is the solution to $u(x,a_m)= e^{-m^2}$.  Note that $a_m\to 0$ as $m\to \infty$. An application of the Borel-Cantelli lemma gives
    \begin{equation*}
        \varlimsup_{a\to 0}  \frac{\tau^x(a)}{u(x,a) \log|\log u(x,a)|} \geq \xi \qquad  (\Pp^x\text{-a.s.}).
    \end{equation*}
    Let $\ell_k$,   $k\geq 1$,  be given by $u(x,\ell_k)=e^{-k}$, $b:=- 4/\log p_{1,\gamma}$, cf.\ \eqref{chung-e45},  and set
    $$
        B_k:= \big\{\omega\,:\,\tau^x(\ell_k,\omega)\geq b u(x,\ell_k) \log|\log u(x,\ell_k)|\big\}.
    $$
    By the upper bound in \eqref{chung-e45} we get
    $$
        \Pp(B_k) \leq  e^{-4 \log|\log u(x,\ell_k)|}\leq \frac{1}{k^4}, \quad k\geq 1,
    $$
    which implies, by the Borel-Cantelli lemma,
    \begin{equation*}
        \varlimsup_{k\to \infty} \frac{\tau^x(\ell_k)}{u(x,\ell_k) \log|\log u(x,\ell_k)|}\leq b.
    \end{equation*}
    Note that by definition the functions $u(x,a)$ and $\tau^x (a)$ are monotone increasing in $a$, and by definition we have $u(x,\ell_{k+1}) =e^{-1} u(x,\ell_k)$.  Therefore,
    \begin{align*}
        \tilde{B}_k
        := \left\{\sup_{\ell_{k+1}\leq a\leq \ell_k} \frac{\tau^x(a)}{u(x,a)\log|\log u(x,a)|}\geq b\right\}
        &\subset \big\{\tau^x(\ell_k) \geq b u(x,\ell_{k+1})\log|\log u(x,\ell_{k+1})|\big\}\\
        &= \big\{\tau^x(\ell_k) \geq b e^{-1} u(x,\ell_k)\log|\log  u(x,\ell_{k+1})|\big\},
    \end{align*}
    implying, by the upper estimate in \eqref{chung-e45}, that  $\Pp(\tilde{B}_k)\leq e^{- 4e^{-1} \log(k+1)}= (k+1)^{-4/e}$.
    Thus,
    \begin{equation}\label{chung-e71}
        \varlimsup_{\ell\to 0} \frac{\tau^x(\ell)}{u(x,\ell) \log|\log u(x,\ell)|}\in [\xi, b].
    \end{equation}

    The expression on the left-hand side of \eqref{chung-e71} belongs to $\mathcal{F}_{0+}$. By the Blumenthal $0$-$1$ law the $\sigma$-algebra $\mathcal{F}_{0+}$ is trivial, implying that there exists a constant $C$ such that
    \begin{equation}\label{chung-e65}
        \varlimsup_{a\to 0} \frac{\tau^x(a) }{u(x,a) \log| \log u(x,a)|} =C  \qquad (\Pp^x\text{-a.s.}).
    \end{equation}
    This constant is the supremum of all $\xi$ such that \eqref{chung-e60} holds. On the other hand, \eqref{chung-e65} is equivalent to
    \begin{equation*}
        \varliminf_{a\to 0} \frac 1a\,\sup_{s\leq u(x,a)\log|\log u(x,a)|}|X_s - x| = C'
        \qquad(\Pp^x\text{-a.s.})
    \end{equation*}
    for some constant $0<C'<\infty$. Substituting $a:= u^{-1}(x,t)$, we get  \eqref{chung-e10}.
\end{proof}

\section{On the upper bound}\label{sec-upper}

In this section we prove \eqref{intro-e02}, that is we give conditions which ensure that there is a norming function $v(t,x)$ with $\varlimsup_{t\to 0} \sup_{s\leq t}|X_s-x|/v(t,x)=0$  $\Pp^x$-a.s.; for a L\'evy process we also obtain conditions ensuring $\varlimsup_{t\to 0} \sup_{s\leq t}|X_s-x|/v(t,x)=\infty$ $\Pp^0$-a.s. For this we adapt Khintchine's criterion \eqref{intro-e50}.
\begin{proposition}\label{upper-10}
    Let $(X_t)_{t\geq 0}$ be a one-dimensional Feller process with symbol $p(x,\xi)$, satisfying \eqref{A1}.  If $v(x,t)\geq 0$ is a function such that $t\mapsto v(x,t)$ is monotone increasing for every $x$ and
    \begin{equation}\label{upper-e10}
        \int\limits_{0+}  \sup_{|y-x|\leq v(x,t)} \psip\big(y, \tfrac{1}{ v(x,t)}\big)\,dt < \infty,
    \end{equation}
    then
    \begin{equation}\label{upper-e12}
        \lim_{t\to 0} \frac{\sup_{s\leq t}|X_s-x|}{v(x,t)}=0
        \qquad(\Pp^x\text{-a.s.}).
    \end{equation}
\end{proposition}
\begin{proof}
    Under our assumptions, the process $(X_t)_{t\geq 0}$ satisfies the maximal inequality \eqref{chung-e31}. As before, we write $X_t^* := \sup_{s\leq t}|X_t - x|$ to simplify notation.

    We will use the (easy direction of the) Borel--Cantelli lemma. Fix some $h\ll 1$ and set $t_k:= h/2^k$.  Then $t_k-t_{k+1}=h/2^{k+1}=t_k/2$. By \eqref{upper-e10} and \eqref{chung-e31} we have
    $$
        \sum_{k=1}^\infty  \Pp^x \big(X_{t_k}^* >v(x,t_k)\big)
        \leq c \sum_{k=1}^\infty t_k \sup_{|y-x|\leq v(x,t_k)} \psip\left(y, \tfrac{1}{v(x,t_k)}\right)
        <\infty.
    $$
    The sum on the right is, up to a multiplicative constant, the lower integral sum for the integral \eqref{upper-e10}. Thus, $\Pp^x \big(X_{t_k}^* \leq v(x,t_k) \text{\ \ for finally all\ \ } k\geq 1\big)=1$.

    Pick $\theta_k\in [t_{k+1},t_k)$. Since $v(x,t)$ is increasing in $t$, we have
    $$
        \Pp^x \big(X_{\theta_k}^*>v(x,\theta_k)\big)
        \leq \Pp^x\big(X_{\theta_k}^*>v(x,t_{k+1})\big)
        \leq c \, \theta_k \sup_{|y-x|\leq v(x,t_{k+1})} \psip\left(y, \tfrac{1}{ v(x,t_{k+1})}\right).
    $$
    Since $\theta_k \leq t_k = 2 t_{k+1}$ we see
    $$
        \sum_{k=1}^\infty \Pp^x \big(X_{\theta_k}^*>v(x,\theta_k)\big)
        <\infty.
    $$
    By the Borel--Cantelli lemma, $\Pp^x\big( X_{\theta_k}^* \leq  v(x,\theta_k) \text{\ \ for finally all\ \ } k\geq 1\big)=1$, implying
    \begin{equation}\label{upper-e14}
        \varlimsup_{t\to 0} \frac{X_t^*}{v(x,t)}\leq 1 \qquad(\Pp^x\text{-a.s.}).
    \end{equation}
    From the definition of $\psip(y,\xi)$, we find $\psip(y, \xi/\lambda) \leq \lambda^{-2} \psip(y,\xi)$ for all $0<\lambda<1$. Thus, \eqref{upper-e10} implies
    \begin{align*}
        \int\limits_{0+}  \sup_{|y-x|\leq \lambda v(x,t)} \psip\big(y, \tfrac{1}{ \lambda v(x,t)}\big)\,dt
        \leq \frac 1{\lambda^2} \int\limits_{0+}  \sup_{|y-x|\leq  v(x,t)} \psip\big(y, \tfrac{1}{  v(x,t)}\big)\,dt
        <\infty.
    \end{align*}
    Because of \eqref{upper-e14} we get
    $$
        \frac{1}{\lambda} \cdot  \varlimsup_{t\to 0} \frac{X_t^*}{v(x,t)}
        = \varlimsup_{t\to 0} \frac{X_t^*}{\lambda v(x,t)}
        \leq 1
        \qquad(\Pp^x\text{-a.s.}).
    $$
    Letting $\lambda\to 0$ gives \eqref{upper-e12}.
\end{proof}

\begin{example}\label{upper-20}
    Suppose that $0<c\leq p(y,\xi)/p(x,\xi)\leq C<\infty$ for all $\xi\in \real$, $|x-y|\leq r$ where $r\ll 1$ is sufficiently small. Then it is enough to check the convergence of the integral
    $$
        \int\limits_{0+} \psip\left(x,\tfrac{1}{v(x,t)}\right) dt < \infty.
    $$
    This integral converges, for example, for functions $v(x,t)$ of the following type
    $\displaystyle
        v(x,t)=\frac{1}{\chi\big(x, \frac 1{t \ell_{\epsilon,n}(t)}\big)}
    $,
    where $\chi(x,\cdot):= [\psip(x,\cdot)]^{-1}$ is the inverse of $\psip(x,\cdot)$, and
    $$
    \ell_{\epsilon,n}(t)
    = |\log t| \cdot |\log |\log t|| \cdot \ldots \cdot \big(\underbrace{\log |\log |  \ldots |\log t |\ldots | }_{n}\big)^{1+\epsilon}
    $$
    for some $\epsilon>0$ and $n\geq 1$.
\end{example}

\begin{example}\label{upper-30}
    Consider the stable-like L\'evy measure from Example~\ref{chung-20}. Since  $\psip(x,\xi)= |\xi|^{\alpha(x)}$, we have an explicit representation of the function $v(x,t)$ from the previous Example \ref{upper-20}:
    \begin{equation}
        v(x,t)= \left( t\ell_{\epsilon,n}(t)\right)^{1/\alpha(x)}.\label{vtx}
    \end{equation}
    Therefore, the integral \eqref{upper-e10} becomes
    \begin{equation}\label{upper-e24}
        \int\limits_{0+} \sup_{|x-y|\leq v(x,t)} \left(\frac{1}{t \ell_{\epsilon,n}(t)}\right)^{\frac{\alpha(y)}{\alpha(x)}} dt <\infty.
    \end{equation}
    Note that $v(x,t)\to 0$ as $t\to 0$.  Since $\alpha$ is continuously differentiable, we can take $t$ so small that in the interval $(x-v(t,x),x+v(t,x))$ there is at most one extremum of $\alpha(y)$. If $\alpha$ has a local maximum at $x$, then the integrand in \eqref{upper-e24} is equal to $\big(t\ell_{\epsilon,n}(t)\big)^{-1}$. Otherwise, $x$ may be a local minimum, or $\alpha'(y)>0$ (respectively, $<0$) on $(x-v(t,x),x+v(t,x))$. In both cases the maximum of $\alpha$ is attained at the end-points of the interval, say, at $x-v(t,x)$.  Using a Taylor expansion, we have
    $$
    \alpha(x- v(x,t))\leq \alpha(x)+ |\alpha' (x-\theta v(x,t))| \,v(t,x),
    $$
    where $\theta=\theta(t,x)\in (0,1)$, implying
    \begin{align*}
        \left(\frac{1}{t \ell_{\epsilon,n}(t)}\right)^{\frac{\alpha(x- v(x,t))}{\alpha(x)}}
        &\leq \frac{1}{t \ell_{\epsilon,n}(t)} \left(\frac{1}{t \ell_{\epsilon,n}(t)}\right)^{\frac{|\alpha'(x- \theta v(t,x))|}{\alpha(x)}\,v(x,t)}\\
        &= \frac{1}{t \ell_{\epsilon,n}(t)} \left(\frac 1{v(x,t)^{v(x,t)}}\right)^{|\alpha'(x- \theta v(t,x))|}
        \leq \frac{C(x)}{t \ell_{\epsilon,n}(t)}
    \end{align*}
    for small $t>0$, where we used that $ v(x,t)^{v(x,t)}\geq 1/2$.

    Thus, in this case Proposition~\ref{upper-10} holds true with $v(x,t)$ as in \eqref{vtx}.
\end{example}

Let us show the counterpart of Proposition~\ref{upper-10}, i.e.\ a condition for $\varlimsup_{t\to 0} |X_t-x|/v(x,t)\geq C$. For this we have to use the direction of the Borel--Cantelli lemma that requires independence. Therefore, we have to restrict ourselves to L\'evy processes. The following proposition appears, with a different proof, already in \cite[Proposition 2.1]{Sa09}, see also \cite[Chapter 7]{Doney}.
\begin{proposition}\label{upper-40}
    Let $(X_t)_{t\geq 0}$ be a pure jump L\'evy process with L\'evy triplet $(0,0,\nu)$ and $v(t)$ be a positive increasing function. If
    \begin{equation}\label{upper-e40}
        \int\limits_{0+}  \nu \big\{y\,:\, |y|\geq 2 C v(t)\big\}\,dt
        =\infty
        \text{\ \ for some\ \ } C>0,
    \end{equation}
    then
    \begin{equation}\label{upper-e42}
        \varlimsup_{t\to 0} \frac{\sup_{s\leq t}|X_s|}{v(t)}
        \geq
        \varlimsup_{t\to 0} \frac{|X_t|}{v(t)}
        \geq
        \frac C3
        \qquad(\Pp^0\text{-a.s.}).
    \end{equation}
\end{proposition}
\begin{proof}
Applying the Etemadi's inequality, cf.\ Billingsley \cite[Theorem 22.5]{billingsley}, we get
\begin{align}\label{lower1}
    3\Pp \big\{|X_t|\geq \tfrac{C}{3} v(t)\big\}
    &\geq  \Pp \Big\{\sup_{s\leq t} |X_s|\geq C v(t)\Big\}
    \geq  1- e^{-t \nu\{y\,:\, |y|\geq 2Cv(t)\}}.
\end{align}
Let now $v(t)$ be such that \eqref{upper-e40} holds true.
There are two possible cases.

\emph{Case 1:} $\varlimsup_{t\to 0} t \nu\{y\,:\, |y|\geq 2Cv(t)\}=0$. Using the inequality $1-e^{-x}\geq c_1 x$ for small $x>0$, we get with \eqref{upper-e40}
$$
    \int\limits_{0+} \frac1t\,\Pp \big\{|X_t|\geq \tfrac{C}{3} v(t)\big\}\, dt
    \geq c_1 \int\limits_{0+} \nu\big\{y\,:\, |y|\geq 2Cv(t)\big\}\,dt
=\infty.
$$

\emph{Case 2:} $\varliminf_{t\to 0} t \nu\{y\,:\, |y|\geq 2Cv(t)\}=c_2>0$. Then
$$
\varliminf\limits_{t\to 0} \left(1- e^{-t \nu\{y\,:\, |y|\geq 2Cv(t)\}}\right)=1- e^{-\varliminf\limits_{t\to 0} t \nu\{y\,:\, |y|\geq 2Cv(t)\}} =1-e^{-c_2}\in (0,1].
$$
 Thus, there exists $t_0$ small enough such that $\Pp \big\{|X_t|\geq \tfrac{C}{3} v(t)\big\}\geq c_3>0$ for all $t\in (0,t_0]$ and we have automatically  $\int\limits_{0+} \frac1t\, \Pp \big\{|X_t|\geq \frac{C}{3} v(t)\big\}\, dt=\infty$.
\end{proof}

\section{LIL results via the symbol of the process}\label{sec-symbol}

In this section we obtain a Chung-type $\liminf$ LIL \eqref{intro-e03} for a Feller process $(X_t)_{t\geq 0}$. We will see that the growth of the norming function $w(x,t)$ is determined by the symbol $p(x,\xi)$ of the process. This result extends, in particular, Proposition~\ref{upper-40}. The method we are presenting here seems to be new also for L\'evy processes, and gives a relatively simple criterion for the norming function, see Remark~\ref{symbol-20}.

Throughout we assume that \eqref{A1} holds with the following stronger version of the sector condition,
\begin{equation*}
    \exists\,c_0\in [0,1)\quad
    \forall x,\xi\in\real \::\: |\Im p(x,\xi)|
    \leq c_0\, \Re p(x,\xi).
\end{equation*}
Note that $c_0 < 1$ means that the drift does not dominate the overall behaviour of the process. For a L\'evy process this implies that a bounded variation process has no drift at all.

We need a further assumption: there exists a monotone increasing function $g$ such that
\begin{equation}\label{A7}\tag{\bfseries A4}
    g(\xi) \leq \Re\,p(x,\xi) \leq  C_p(1+|\xi|^2), \quad x\in \real, \; |\xi|\geq 1.
\end{equation}
We also need the following estimate for the characteristic function $\lambda_t(x,\xi)=\Ee^x e^{i\xi (X_t-x)}$ which is due to \cite[Proposition~2.4]{SW12}:
\begin{equation}\label{symbol-e05}
    \sup_{x\in \real} |\lambda_t(x,\xi)|\leq \exp\big[ -\delta t \inf_{x\in \real} \Re p(x,\xi)\big], \quad t\in [0,t_0],\quad t_0=t_0(\xi, \epsilon),
\end{equation}
where $\delta=\delta(c)=1-c_0-\epsilon>0$,  and $0\leq c_0<1$ is the sector constant.

\begin{remark} \label{rem-cont}
    \eqref{A7} ensures that the function $t_0(\xi,\epsilon)$ is continuous in $\xi$. This follows from the proof of \cite[Proposition~2.4]{SW12}. The upper bound in \eqref{A7} means that the generator $A=-p(x,D)$ has \emph{bounded coefficients}, cf.\ \cite{schilling-schnurr} for details; in fact, $C_p= 2 \sup_{x\in\real} \sup_{|y|\leq 1} |p(x,\eta)|$.
\end{remark}
\begin{theorem}\label{symbol-10}
    Let $(X_t)_{t\geq 0}$ be a Feller process such that $X_0 = x$ and with the symbol $p(x,\xi)$ satisfying the conditions \eqref{A1} and \eqref{A7} with a sector constant $c_0\in [0,1)$. Let $w(x,t)$, $t>0, x\in\real$, be a positive function which is for all $x$  monotone decreasing as a function of $t$. Then we have $\Pp^x$-a.s.
    \begin{equation*}
    \varliminf_{t\to 0}\frac{|X_t-x|}{w(x,t)} =
     \left\{\begin{aligned}
     &\infty\geq \gamma(x)>0,
     \\[-\medskipamount]
     &\infty
     \end{aligned}\right.
    \quad  \text{according to} \quad
    \varliminf_{t\to 0} t g\big(\tfrac 1{w(x,t)}\big) =
     \left\{\begin{aligned}
     &c(x) > 0, \\[-\medskipamount]
     &\infty.
     \end{aligned}\right.
    \end{equation*}
\end{theorem}

\begin{proof}
    Take $1<a<b<\infty$. By Fubini's theorem we find
    \begin{align*}
        \left| \Ee^x \int_a^b e^{i\xi\,\frac{X_t-x}{w(x,t)}}\,d\xi\right|
        = \left| \int_a^b \lambda_t \big(x,\tfrac{\xi}{w(x,t)}\big)\,d\xi\right|
        &\leq \int_a^b \exp\big[-\delta t g(\tfrac{\xi}{w(x,t)}) \big]\,d\xi\\
        &\leq (b-a)  \exp\big[-\delta t g(\tfrac{1}{w(x,t)}) \big],
    \end{align*}
    where we  used the monotonicity of $g$ in the last estimate. This inequality holds for all $0\leq t\leq t(a,b,\epsilon)$, $t(a,b,\epsilon)=\inf_{\xi\in[a,b]} t_0(\xi,\epsilon)$ where $t_0(\xi,\epsilon)$ is the constant from \eqref{symbol-e05}. Since it depends continuously on $\xi$, cf.\ Remark~\ref{rem-cont}, we have $t(a,b,\epsilon)>0$. Taking the $\varlimsup_{t\to 0}$ on both sides, we get
    \begin{align}\label{symbol-e16}
        \varlimsup_{t\to 0} \left| \Ee^x \int_a^b e^{i\xi\,\frac{X_t-x}{w(x,t)}}\,d\xi\right|
        \leq (b-a)\exp\Big\{-\delta \varliminf_{t\to 0} \Big[t g\big(\tfrac{1}{w(x,t)}\big)\Big] \Big\}.
    \end{align}

\paragraph{Case 1}
    Assume that
    $
        \varliminf_{t\to 0} t g(1/w(x,t)) = c(x) > 0
    $.
    Then
    \begin{equation*}
        \varlimsup_{t\to 0} \left| \Ee^x \int_a^b e^{i\xi\,\frac{X_t-x }{w(x,t)}}\,d\xi\right|
        \leq (b-a) e^{-\delta c(x)}.
    \end{equation*}
   On the other hand, using $|z|\geq |\Re\, z|$, we derive
   \begin{align*}
   \left| \int_a^b e^{i\xi\,\frac{X_t-x }{w(x,t)}}\,d\xi\right|
   \geq   \left| \int_a^b \cos\left( \xi\,\frac{X_t-x }{w(x,t)}\right) d\xi\right|
   \geq   \int_a^b \cos\left( \xi\,\frac{X_t-x }{w(x,t)}\right) d\xi.
   \end{align*}
    Suppose that the claim does not hold, and $ \varliminf_{t\to 0} {|X_t-x|}/{w(x,t)}=0$. Without loss of generality, we can choose $a$ and $b$ such that  $\cos\big( \xi\,\frac{X_t-x }{w(x,t)}\big)>0$ for $a<\xi<b$. Since $\cos$ is bounded below by $-1$, we can apply Fatou's lemma and get
    $$
    \varliminf_{t\to 0}  \left| \Ee^x \int_a^b e^{i\xi\,\frac{X_t-x }{w(x,t)}}\,d\xi\right|\geq
    \Ee^x \left( \int_a^b  \varliminf_{t\to 0}   \cos\left( \xi\,\frac{X_t-x }{w(x,t)}\right) d\xi\right)= b-a.
    $$
    Thus, we arrive at  $1\leq  e^{-\delta c(x)}$, which is wrong, since $c(x)$ is strictly positive.

\paragraph{Case 2}
    Assume that
    $
         \varliminf_{t\to 0} t g(1/w(x,t)) =\infty
    $.
    From \eqref{symbol-e16} and the fact that $|\Re z| \leq |z|$ we see
    \begin{equation}\label{symbol-e17}
    0= \varliminf_{t\to 0}  \left| \int_a^b \cos\left( \xi\,\frac{X_t-x }{w(x,t)}\right) d\xi\right|.
    \end{equation}
    Assume that there exists a sequence $(t_n)_{n\geq 0}$ with $\lim_{n\to\infty} t_n= 0$ and $\lim_{n\to \infty}  {|X_{t_n}-x|}/{w(x,t_n)}=c<\infty$. Since $1<a<b<\infty$ are arbitrary, we can chose the interval $[a,b]$ in such a way that
    $$
        \int_a^b \cos\left( \xi\,\frac{X_{t_n}-x }{w(x,t_n)}\right) d\xi>\epsilon >0
        \quad\text{for all\ \ } n\geq 1
    $$
    and some $\epsilon=\epsilon(c)>0$.
    This contradicts \eqref{symbol-e17} and the proof is finished.
\end{proof}

\begin{remark}\label{symbol-20}
If the constant $c(x)$ appearing in the statement of the preceding theorem is uniformly bounded away from zero, i.e.\ $\inf_x c(x)=c>0$, then $\inf_x\gamma(x)=\gamma>0$. Indeed, assume that $\gamma=0$. Taking $\sup_x$ on both sides of \eqref{symbol-e16}  we get in the same way as above that
\begin{align*}
    (b-a)e^{-c\delta}
    \geq  \sup_x  \varlimsup_{t\to 0} \left| \Ee^x \int_a^b e^{i\xi\,\frac{X_t-x}{w(x,t)}}\,d\xi\right|
    \geq \Ee^x \left( \int_a^b \inf_x \varliminf_{t\to 0}   \cos\left( \xi\,\frac{X_t-x }{w(x,t)}\right) d\xi\right)= b-a,
\end{align*}
which contradicts to the assumption $c>0$.
\end{remark}

\begin{remark}\label{symbol-22}
If $(X_t)_{t\geq 0}$ is a symmetric L\'evy process with characteristic exponent $\psi(\xi)\geq g(\xi)\geq 0$ and a monotone increasing function $g$, Theorem \ref{symbol-10} reads    \begin{equation*}
    \varliminf_{t\to 0}\frac{|X_t|}{w(t)} =
     \left\{\begin{aligned}
     &\infty > \gamma>0,
     \\[-\medskipamount]
     &\infty
     \end{aligned}\right.
    \quad  \text{according to} \quad
    \varliminf_{t\to 0} t g\big(\tfrac 1{w(t)}\big) =
     \left\{\begin{aligned}
     &c > 0, \\[-\medskipamount]
     &\infty.
     \end{aligned}\right.
    \end{equation*}
Indeed: Now we can take $a=0$ and $b=1$ and get
$$
    \left|\Ee \left[\frac{e^{i\,\frac{X_t}{w(t)}}-1}{X_t/w(t)}\right]\right|
    =\left|\Ee\left[ \int_0^1 e^{-i\xi\,\frac{X_t}{w(t)}}\,d\xi\right]\right|
    =\Ee \left[\int_0^1 e^{-i\xi\,\frac{X_t}{w(t)}}\,d\xi\right]
    = \int_0^1 e^{- t\psi(\frac\xi{w(t)})}\,d\xi.
$$
Assume in Case 1 of the proof of Theorem \ref{symbol-10} that $\varliminf_{t\to 0} t\psi(1/w(t))= c\in (0,\infty)$ and $\varliminf_{t\to 0} |X_t|/w(t)=\infty$. Let $(t_n)_{n\geq 0}$ be a sequence decreasing to $0$ such that $\lim_{n\to\infty } t_n\psi(1/w(t_n))=c$. Then
$$
    \lim_{n\to\infty}\left|\Ee \left[\frac{e^{i\,\frac{X_{t_n}}{w(t_n)}}-1}{X_{t_n}/w(t_n)}\right]\right|
    = \lim_{n\to\infty}\int_0^1 e^{- t_n\psi(\frac\xi{w(t_n)})}\,d\xi
    = \int_0^1 \lim_{n\to\infty} e^{- t_n\psi(\frac\xi{w(t_n)})}\,d\xi
    \geq  \lim_{n\to\infty} e^{- t_n g(\frac 1{w(t_n)})}
    = e^{-c}.
$$
From the elementary estimate $|e^{i\xi}-1|\leq |\xi|$ we see that the expression on the left tends to $0$, and we have reached a contradiction also in this case. The rest of the proof applies literally.
\end{remark}






\begin{thebibliography}{00}

\bibitem{ADS12}
    F.\ Aurzada, L.\ D\"oring, M.\ Savov:  Small time Chung type LIL for L\'evy processes. To appear in \emph{Bernoulli}.

\bibitem{bass}
    R.F.\ Bass: Uniqueness in law for pure jump type Markov processes. \emph{Probab.\ Theory Rel.\ Fields} \textbf{79} (1988), 271--287.

\bibitem{billingsley}
    P.\ Billingsley: \emph{Probability and Measure (Third Edition)}. Wiley, New York 1995.

\bibitem{Doney} R.\ A.\ Doney: \emph{Fluctuation Theory for L\'evy Processes}. \'Ecole d'\'Et\'e de Probabilit\'es de Saint-Flour XXXV, Lecture Notes in Mathematics vol.\ \textbf{1897}. Springer, Berlin 2007.


\bibitem{Du74}
    C.\ Dupuis: Mesure de Hausdorff de la trajectoire de certains processus \`a accroissements ind\'ependants et stationnaires. In: \emph{S\'eminaire de Prohabilit\'es VIII (1972/73)}. Springer, Lecture Notes in Mathematics \textbf{381}, Berlin 1974, 40--77.

\bibitem{FP71}
    B.\ Fristedt, W.\ Pruitt: Lower functions for increasing random walks and subordinators. \emph{Z.\ Wahrscheinlichkeitstheor.\ verw.\ Geb.} \textbf{18} (1971) 167--182.

\bibitem{FP72}
    B.\ Fristedt, W.\ Pruitt: Uniform lower functions for subordinators. \emph{Z.\ Wahrscheinlichkeitstheor.\ verw.\ Geb.} \textbf{24} (1972) 63--70.

\bibitem{Ha71}
    J.\ Hawkes: A lower Lipschitz condition for the stable subordinator. \emph{Z.\ Wahrscheinlichkeitstheor.\ verw.\ Geb.} \textbf{17} (1971) 23--32.

\bibitem{Jac} N.\ Jacob: \emph{Pseudo-Differential Operators and Markov Processes (3 Vols.)}. Imperial College Press, London, 2001--05.

\bibitem{jac-sch-survey}
    N.\ Jacob, R.L.\ Schilling: L\'{e}vy-type processes and pseudo differential operators. In: O.\ E.\ Barndorff-Nielsen et al.\ (eds.): \emph{L\'evy Processes: Theory and Applications}, Birkh\"auser, Boston 2001, 139--168.

\bibitem{Kh36}
    A.I.\ Khintchine: Asymptotic laws in probability theory. (Asimptoticheskie zakony teorii veroyatnostei). ONTI, Moscow--Leningrad, 1936. (In Russian)

\bibitem{Kh39}
    A.I.\ Khintchine: Sur la croissance locale des processus stochastiques homogen\`es \`a accroissements independants (Russian, French summary), \emph{Izvestia Akad.\ Nauk SSSR, Ser. Math.} \textbf{3} (1939), 487--508.

\bibitem{Ko00}
    V.N.\ Kolokoltsov: Symmetric stable laws and stable-like jump-diffusions. \emph{Proc.\ London Math.\ Soc.} \textbf{80} (2000), 725--768.

\bibitem{Pr81}
    W.E.\ Pruitt: The growth of random walks and L\'evy processes. \emph{Ann.\ Probab.} \textbf{9} (1981), 948--956.

\bibitem{PT69}
    W.E.\ Pruitt,  S.J.\ Taylor: Sample path properties of processes with stable components. \emph{Z.\ Wahrscheinlichkeitstheor.\ verw.\ Geb.} \textbf{12} (1969) 267--289.

\bibitem{Sa09}
    M.\ Savov: Small time two-sided LIL behavior for L\'evy processes at zero. \emph{Probab.\ Theor.\ Rel.\ Fields} \textbf{144} (2009), 79--98.

\bibitem{Sch98}
    R.L.\ Schilling: Growth and H\"{o}lder conditions for the sample paths of Feller processes. \emph{Probab.\ Theor.\ Rel.\ Fields} \textbf{112} (1998), 565--611.

\bibitem{schilling98pos}
    R.L.\ Schilling: Conservativeness and extensions of Feller semigroups. \emph{Positivity} \textbf{2} (1998), 239--256.

\bibitem{schilling-schnurr}
    R.L.\ Schilling, A.\ Schnurr: The symbol associated with the solution of a stochastic differential equation. \emph{El.\ J.\ Probab.} \textbf{15} (2010), 1369--1393.

\bibitem{SW12} R.L.\ Schilling, J.\ Wang: Some theorems on Feller processes: transience, local times and ultracontractivity. \emph{Trans.\ Am.\ Math.\ Soc.} \textbf{365} (2013) 3255--3286.


\bibitem{Ta67}
    S.J.\ Taylor: Sample path properties of a transient stable process. \emph{J.\ Math.\ Mech.} \textbf{16} (1967) 1229--1246.

\end{thebibliography}



\end{document}